%% file: main.tex
\documentclass[a4paper,12pt]{amsart}

\input{Preamble}
\usepackage{verbatim}

\begin{document}

\title[]{Classifying spaces for families of abelian subgroups of braid groups, RAAGs and graphs of abelian groups.
\mycomment{On the dimension of
classifying spaces for families of abelian subgroups of braid groups, RAAGs and graphs of abelian groups.}
\mycomment{Calculating the dimension of
classifying spaces for families of abelian subgroups of braid groups, RAAGs, and graphs of groups of finitely generated  abelian groups}}

\author[Porfirio L. León Álvarez ]{Porfirio L. León Álvarez}

\address{Instituto de Matemáticas, Universidad Nacional Autónoma de México. Oaxaca de Juárez, Oaxaca, México 68000}
\email{porfirio.leon@im.unam.mx}


\date{}



\begin{abstract}
Given a group $G$ and an integer $n\geq 0$ we consider the
family $\calF_n$ of all virtually abelian subgroups of $G$ of $\rank$ at most $n$. In this  article we prove that for each $n\ge2$ the Bredon cohomology, with respect to the family $\calF_n$,  of a free abelian group with  $\rank$  $k > n$ is nontrivial in dimension $k+n$; this answers a question of Corob Cook, Moreno, Nucinkis and Pasini \cite[Question 2.7]{victor:nucinkis:corob}. As an application,   we compute the minimal dimension of a classifying space for the family $\calF_n$ for  braid groups, right-angled Artin groups, and  graphs of groups whose vertex groups are infinite finitely generated virtually abelian groups, for all $n\ge2$. The main tools that we use are the Mayer-Vietoris sequence for Bredon cohomology,  Bass-Serre theory,  and the Lück-Weiermann construction.
\end{abstract}
\maketitle

\section{Introduction}
\noindent Given a group $G$, we say that a collection $\calF$ of subgroups of $G$ is a \emph{family} if it is non-empty and closed under conjugation and taking subgroups. We fix a group $G$ and a family $\calF$ of subgroups of $G$. We say that a $G$-CW-complex $X$ is a \emph{model for the classifying space} $E_{\calF}G$  if  all of its isotropy groups belong to $\calF$  and if $Y$ is a $G$-CW-complex with isotropy  groups belonging to $\calF$, there is precisely one $G$-map $Y\to X$ up to $G$-homotopy.  It can be shown that a model for the classifying space   $E_{\calF}G$ always exists and it is unique up to $G$-homotopy equivalence. We define the \emph{$\calF$-geometric dimension} of $G$   as$$\gd_{\calF}(G)=\min\{n\in \mathbb{N}| \text{ there is a model for } E_{\calF}G \text{ of dimension } n \}.$$

The $\calF$-geometric dimension has its algebraic counterpart, the \emph{$\calF$-cohomological dimension} $\cd_{\calF}(G)$, which can be defined in terms of Bredon cohomology (see Section 2). The $\calF$-geometric dimension and the $\calF$-cohomological dimension satisfy the following inequality (see \cite[Theorem 0.1]{LM00}):
\[\cd_{\calF}(G)\leq \gd_{\calF}(G)\leq \max\{\cd_{\calF}(G), 3\}.\]
It follows that if $\cd_{\calF}(G)\geq 3$ then $\cd_{\calF}(G)=\gd_{\calF}(G)$. It is not generally true that $\cd_{\calF}(G)=\gd_{\calF}(G)$. For example, for the family of finite subgroups $\calF_0$, in \cite{NB} it was proved that there is a right-angled Coxeter group $W$ such that $\cd_{\calF_0}(W)=2$ and $\gd_{\calF_0}(W)=3$. For other examples see \cite{SS19}.

 Let $n\geq0$ be an integer. A group is said to be \emph{virtually $\mathbb{Z}^n$} if it contains a subgroup of finite index isomorphic to $\mathbb{Z}^n$. Define the family \[\calF_n=\{H\leq G| H \text{ is virtually } \mathbb{Z}^r \text{ for some 
} 0\le r \le n \}.\] The  families  $\calF_0$ and $\calF_1$  are relevant due to their connection with the Farrell-Jones and Baum-Connes isomorphism conjectures; see for example \cite{LR05}. The Farrell-Jones conjecture has been proved for braid groups in \cite{MR1797585,MR1759505,MR3598161} and for some even Artin groups in \cite{MR4360030}.

For $n\geq 2$, the families $\calF_n$ have been recently studied by
several people;  see for example \cite{tomasz,HP20,MR4472883,SS19,Rita:Porfirio:Luis}. For a virtually $\Z^n$ group $G$ it was proved in \cite{tomasz} that $\gd_{\calF_k}(G)\leq n+k$ for all $0\leq k<n$. For a free abelian group this upper bound was also obtained by Corob Cook, Moreno, Nucinkis and Pasini in \cite{victor:nucinkis:corob}  and they asked whether this upper bound was sharp:
\begin{question} \cite[Question 2.7]{victor:nucinkis:corob}
For $0\leq k< n$, is $\gd_{\calF_k}(\Z^n)=n+k$?
\end{question}
We answer this question affirmatively in \Cref{Virtually:dimension:Z}. For $k=1$, this was proved in \cite[Theorem 5.13]{LW12} and for $k=2$ in \cite[Proposition A.]{Onorio}. 
As an application, we provide lower bounds for the $\calF_k$-geometric dimension  of virtually abelian groups, braid groups, and right-angled Artin groups (RAAGs). Combining these lower bounds with previously known results in the literature, we show that they are sharp. We also prove that the $\calF_k$-geometric dimension is equal to the $\calF_k$-cohomological  dimension in all these cases. On the other hand, inspired by 
 \cite{MR4472883}, we use Bass-Serre theory to explicitly calculate, for all $k\geq 1$, the $\calF_k$-geometric dimension of graphs of groups whose vertex groups are infinite finitely generated virtually abelian groups.

There are few explicit calculations of the $\calF_n$-geometric dimension for $n\geq 2$. For example, the $\calF_n$-geometric dimension for orientable $3$-manifold groups was explicitly calculated in \cite{MR4472883} for all $n\geq 2$. In \cite[Proposition A.]{Onorio}, it was shown that $\gd_{\calF_2}(\Z^k)=k+2$ for all $k\geq 3$. With our results we  add  braid groups, RAAGs, and  graphs of groups whose vertex groups are infinite finitely generated virtually abelian groups to this list. In what follows, we present more precisely these results.

\subsection*{The $\calF_n$-dimension of virtually abelian groups}

Let $G$ be a virtually $\Z^n$ group.
In \cite[Proposition 1.3]{tomasz},  it was proved that $\gd_{\calF_k}(G)\leq n+k$ for $0\leq k< n$. For a free abelian group this upper bound has also been proved in \cite{victor:nucinkis:corob}. In this article, we prove that this upper bound is sharp.

\begin{theorem}\label{Virtually:dimension:Z}
Let $k,n\in \mathbb{N}$ such that $0\leq k< n$. Let $G$ be a virtually $\Z^n$ group. Then  $\gd_{\calF_k}(G)=\cd_{\calF_k}(G)=n+k$.
\end{theorem}
For $k=1$, the \cref{Virtually:dimension:Z} was proved in \cite[Theorem 5.13]{LW12}. For $k=2$, a particular case was proved in \cite[Proposition A.]{Onorio}, specifically $\gd_{\calF_2}(\Z^k)=k+2$ for all $k\geq 3$. 
As a corollary of \cref{Virtually:dimension:Z} we have 
\begin{corollary}\label{subgroup:Z:lower:bound:1}
Let $n\geq 1$ and let $G$ be a group that has a virtually $\Z^n$ subgroup. Then for $0\leq k<n$ we have $\gd_{\calF_k}(G)\geq n+k$ and $\cd_{\calF_k}(G)\geq n+k$.  
\end{corollary}

\subsection*{The $\calF_n$-dimension of braid groups}
There are various ways to define the (full) braid group $B_n$ on $n$ strands. For our purposes, the following definition is convenient. Let $D_n$ be the closed disc with $n$ punctures. We define the \textit{braid group $B_n$ on $n$ strands} as the isotopy classes of orientation preserving diffeomorphisms of $D_n$ that restrict to the identity on the boundary $\partial D_n$. In the literature, this group is known as the mapping class group of $D_n$. It is well known that  $\gd_{\calF_0}(B_n)=n-1$ see for example \cite[Section 3]{Arnold2014}. In \cite[Theorem 1.4]{Rita:Porfirio:Luis}, it was proved that $\gd_{\calF_k}(B_n)\leq n+k-1$ for all $k\in \N$. Using \cref{subgroup:Z:lower:bound:1} and \cite[Proposition 3.7]{Ramon:Juan}  we prove that this upper bound is sharp.
\begin{theorem}\label{gd:braid:groups}
    Let  $k,n \in \N$ such that $0\leq k< n-1$ and $G$ be either the full braid group  $B_n$ or the pure braid group $P_n$. Then  $\gd_{\calF_k}(G)=\cd_{\calF_k}(G)=\vcd(G)+k= n+k-1.$
\end{theorem}
\subsection*{The $\calF_n$-dimension of right-angled Artin groups} Let $\Gamma$ be a finite simple graph, i.e. a finite graph without loops or multiple edges between vertices. We define the \textit{right-angled Artin group} (RAAG) $A_{\Gamma}$ as the group generated by the vertices of $\Gamma$ with all the relations of the form $vw=wv$ whenever $v$ and $w$ are joined by an edge.

Let  \(A_\Gamma\) be a RAAG. It is well-known $A_\Gamma$ is a $\cat(0)$ group, in fact $A_\Gamma$ acts on the universal cover  \(\Tilde{S}_\Gamma\) of its  Salvetti CW-complex \(S_\Gamma\), see Section 4.2. In \cite{tomasz} it was proved that  \(\cd_{\calF_k}(A_\Gamma) \leq \dim(S_{\Gamma})+k+1\). Following the proof of \cite[Proof of Theorem 3.1]{tomasz} and using \cite[Proposition 7.3]{HP20}, we can actually show that \(\cd_{\calF_k}(A_\Gamma) \leq \dim(S_{\Gamma})+k\) in \cref{upper:bound:cd:RAAG}. Moreover, by using \cref{subgroup:Z:lower:bound:1} and \cref{subgroup:abelian:rank:vcd:salvetti}, we can prove that this upper bound is sharp.

 \begin{theorem}\label{Virtually:abelian:RAAG}
Let $A_\Gamma$ be a right-angled Artin group.  Then for $0 \leq k < \cd(A_\Gamma)$ we have
$ \gd_{\calF_k}(A_\Gamma)= \cd_{\calF_k}(A_\Gamma)=\dim(S_{\Gamma})+k=\cd(A_\Gamma)+k$.
\end{theorem}
This calculation of the $\calF_k$-geometric dimension of a RAAG $A_\Gamma$ is explicit because the dimension of the Salvetti CW-complex $S_{\Gamma}$ is the maximum of all natural numbers $n$ such that there is a complete subgraph $\Gamma'$ of $\Gamma$ with $|V(\Gamma')| = n$  (see \cref{subgroup:RAAG:vcd}).

Using \cref{subgroup:Z:lower:bound:1} we can give a lower bound for the $\calF_k$-geometric dimension of the outer automorphism group $\Out(A_\Gamma)$ of some RAAGs $A_\Gamma$.
\begin{proposition}\label{gd:out:fn}
    Let $n\geq 2$. Let $F_n$ be the free group in $n$ generators. Then for all $0\leq k< 2n-3$ we have $$\gd_{\calF_k}(\Out(F_n))\geq \vcd(\Out(F_n))+k\geq 2n+k-3.$$ 
\end{proposition}

\begin{proposition}\label{gd:diamonds}
  Let $A_d$ be the right-angled Artin group given by a string of $d$ diamonds. Then $\gd_{\calF_k}(\Out(A_d))\geq\vcd(\Out(A_d))+k\geq 4d +k-1$ for all $0\leq k< 4d-1$. 
\end{proposition}

\begin{question}
Given \cref{gd:braid:groups} and \cref{virtually:abelian:RAAG}, it is natural to ask whether it is true that in \cref{gd:out:fn} we can have $\gd_{\calF_k}(\Out(F_n)) \leq \vcd(\Out(F_n)) + k \leq2n + k - 3$ for all $0 \leq k < 2n - 3$. Similarly, if it is true that in \cref{gd:diamonds} we can have $\gd_{\calF_k}(\Out(A_d)) \leq \vcd(\Out(A_d)) + k \leq 4d + k - 1$ for all $0 \leq k < 4d - 1$.
\end{question}
\subsection*{The $\calF_n$-geometric dimension for graphs of groups of finitely generated virtually abelian groups.}

 Inspired by 
 \cite{MR4472883}, we use   Bass-Serre theory, \cref{Virtually:dimension:Z} and \cref{subgroup:Z:lower:bound:1} to compute the $\calF_n$-geometric dimension of graphs of groups whose vertex groups are finitely generated virtually abelian groups. 
\begin{theorem}\label{graph:groups:1}
    Let $Y$ be a finite graph of groups such that for each $v\in V(Y)$ the group $G_v$ is infinite finitely generated virtually abelian,  with $\rank(G_e)<\rank(G_v)$. Suppose that the splitting of  $G=\pi_1(Y)$ is acylindrical. Let  $m=\max\{\rank(G_v)| v\in V(Y) \}$. Then for $1\leq k<m$ we have $\gd_{\calF_k}(G)=m+k.$ 
\end{theorem}

\begin{corollary}\label{gd:product:free}
 Let $Y$ be a finite graph of groups such that for each $v\in V(Y)$ the group $G_v$ is 
 infinite finitely generated virtually abelian and for each $e\in E(Y)$ the group $G_e$ is a finite group.  Let  $m=\max\{\rank(G_v)| v\in V(Y) \}$. Then for $1\leq k<m$ we have $\gd_{\calF_k}(G)=m+k.$ 
\end{corollary}
\subsection*{Outline of the paper.} 
In Section 2, we introduce the Lück--Weiermann construction, which enables us to build models inductively for the classifying space of $E_{\calF_n\cap H}\Z^n$. Later in the same section, we define Bredon cohomology and present the Mayer-Vietoris sequence, which is as a crucial tool in proving \cref{cd:lower:bound}. In \cref{sec:3}, we  prove \cref{Virtually:dimension:Z}. In \cref{sec:4}, we  present some applications of \cref{subgroup:Z:lower:bound:1}, for instance, we explicitly calculate the  $\calF_k$-geometric dimension of  braid groups and RAAGs.  Furthermore, we  provide a lower bound for the $\calF_k$-geometric dimension of the outer automorphism group of certain RAAGs. Finally, we use Bass-Serre theory to prove \cref{graph:groups:1}. 

\subsection* {Acknowledgements}I was supported by a doctoral scholarship of the Mexican Council of Humanities, Science and Technology (CONAHCyT). I would like to thank Luis Jorge Sánchez Saldaña for several useful discussions during the preparation of this article. I also thank Rita Jiménez Rolland for comments on a draft of the present article. I am  grateful for the financial support of DGAPA-UNAM grant PAPIIT IA106923
and CONACyT grant CF 2019-217392. I thank the anonymous referee for corrections and comments that improved the exposition.

\section{Preliminaries}\label{sec:2}

\subsection*{The Lück-Weiermann construction}
In this subsection, we give a particular construction of Lück-Weiermann \cite[Theorem 2.3]{LW12} that we will use later.
\begin{definition}
 Let $\calF\subset\calG$ be two families of subgroups of $G$. Let $\sim$ be  an equivalence relation in $\calG-\calF$. We say that $\sim$ is strong  if the following is satisfied 
 
\begin{enumerate}[a)]
    \item If  $H, K \in \calG-\calF$ with $H\subseteq K$, then $H\sim K$;
    \item If $H, K \in \calG-\calF$ and $g\in G$, then $H\sim K$ if and only if $gHg^{-1} \sim gKg^{-1}$.
\end{enumerate}
 \end{definition}
 \begin{definition}
     Let $G$ be a group and $L, K$ be subgroups of $G$. We say that $L$ and $K$ are commensurable if $L\cap K$ has finite index in both $L$ and $K$.
 \end{definition}
\begin{definition}
Let $G$ be a group and let $H$ be a subgroup of $G$. We define the commensurator of $H$ in $G$ as  $$N_{G}[H]:=\{g\in G | gHg^{-1}\text{ is commensurable with }H \}.$$
\end{definition} 
\begin{definition}\label{restr:family}
Let $G$ be a group, let $H$ be a subgroup of $G$, and $\calF$ a family of subgroups of $G$. We define the family $\calF\cap H$ of $H$ as all the subgroups of $H$ that belong to $\calF$. We can complete the family $\calF\cap H$ in order to get a family $\overline{\calF\cap H}$  of  $G$. 
\end{definition}

 \begin{remark}\label{families}Following the notation of \cref{restr:family} note that:
 \begin{itemize}
     \item If $H=G$ then  $\overline{\calF\cap H}=\calF$.
     \item If $H$ is normal subgroup of $G$, then $\overline{\calF\cap H}=\calF\cap H$.
 \end{itemize}
 \end{remark}
 
      Let $G$ be a group, $H$ a subgroup of $G$ and $n\geq0$. Consider the following nested families of $G$,  $\overline{\calF_n\cap H}\subseteq \overline{\calF_{n+1}\cap H}$, let $\sim$ the equivalence relation in $\overline{\calF_{n+1}\cap H}-\overline{\calF_n\cap H}$ given by commensurability. It is easy to check that this is a strong equivalence relation.

We introduce the following notation:
\begin{itemize}
    \item We denote by $(\overline{\calF_{n+1}\cap H}-\overline{\calF_n\cap H})/\sim$ the equivalence classes  in  $\overline{\calF_{n+1}\cap H}-\overline{\calF_n\cap H}$. Given $L\in (\overline{\calF_{n+1}\cap H}-\overline{\calF_n\cap H})$ we denote by $[L]$ its equivalence class.
    
    \item Given $[L]\in (\overline{\calF_{n+1}\cap H}-\overline{\calF_n\cap H})/\sim$, we define the next family of subgroups of  $N_{G}[L]$ $$(\overline{\calF_{n+1}\cap H})[L]:=\{K\leq N_G[L]| K\in  (\overline{\calF_{n+1}\cap H}-\overline{\calF_n\cap H}), [K]=[L]\}\cup (\overline{\calF_n\cap H}\cap N_G[L]).$$
    
\end{itemize}

 \begin{theorem}\cite[Theorem 2.3]{LW12}\label{LW} Let $G$ be a group, let $H$ be a subgroup of $G$ and $n\geq0$. Consider the following nested families of $G$,  $\overline{\calF_n\cap H}\subseteq \overline{\calF_{n+1}\cap H}$, let $\sim$ be the equivalence relation given by commensurability in $\overline{\calF_{n+1}\cap H}-\overline{\calF_n\cap H}$. 
 Let $I$ be a complete set of representatives of conjugation classes in $(\overline{\calF_{n+1}\cap H}-\overline{\calF_n\cap H})/\sim$. Choose arbitrary $N_G[L]$-CW-models for $E_{(\overline{\calF_n\cap H})\cap N_{G}[L]}N_{G}[L]$ and $E_{ (\overline{\calF_{n+1}\cap H})[L]}N_{G}[L]$, and an arbitrary model for  $E_{\overline{\calF_n\cap H}}G$. Consider the  following $G$-push-out 
  \[
\begin{tikzpicture}
  \matrix (m) [matrix of math nodes,row sep=3em,column sep=4em,minimum width=2em]
  {
     \displaystyle\bigsqcup_{[L]\in I} G\times_{N_{G}[L]}E_{(\overline{\calF_n\cap H})\cap N_{G}[L]}N_{G}[L] & E_{\overline{\calF_n\cap H}}G \\
      \displaystyle\bigsqcup_{[L]\in I} G\times_{N_{G}[L]}E_{ (\overline{\calF_{n+1}\cap H})[L]}N_{G}[L] & X \\};
  \path[-stealth]
    (m-1-1) edge node [left] {$\displaystyle\bigsqcup_{[L]\in I}id_{G}\times_{N_G[L]}f_{[L]}$} (m-2-1) (m-1-1.east|-m-1-2) edge  node [above] {$i$} (m-1-2)
    (m-2-1.east|-m-2-2) edge node [below] {} (m-2-2)
    (m-1-2) edge node [right] {} (m-2-2);
\end{tikzpicture}
\]
such that $f_{[L]}$ is a cellular $G$-map for every $[L]\in I$ and either (1) $i$ is an inclusion of $G$-CW-complexes, or (2) such that every map $f_{[L]}$ is an inclusion of $G$-CW-complexes for every $[L]\in I$ and $i$ is a cellular $G$-map. Then $X$ is a model for $E_{\overline{\calF_{n+1}\cap H}}G$.

 \end{theorem}
 
\begin{remark}\label{dimension:LW}
The conditions in \cref{LW} are not restrictive. For instance, to satisfy the  condition $(2)$, we can use the equivariant cellular approximation theorem to assume that the maps $i$ and $f_{[L]}$  are cellular maps for all $[L]\in I$,  and to make the function $f_{[L]}$ an inclusion for every $[L]\in I$,  we can replace the spaces by the mapping cylinders. See \cite[Remark 2.5]{LW12}. 
\end{remark}
Following the notation from \cref{LW} we have  
\begin{corollary}\label{lw:gd:upper:bound}
$\gd_{\overline{\calF_{n+1}\cap H}}(G)\leq \max\{ \gd_{ \overline{\calF_{n}\cap H}}(G)+1, \gd_{ (\overline{\calF_{n+1}\cap H})[L]}(N_{G}[L])| L\in I\}.$
\end{corollary}
 
 \subsection*{The push-out of a union of families}
 The following lemma will be also useful.
\begin{lemma}\cite[Lemma~4.4]{DQR11}\label{lemma:union:families}
Let $G$ be a group and $\calF,\, \calG$ be two families of subgroups of $G$.  Choose arbitrary $G$-$CW$-models for $E_{\calF} G$, $E_{\calG}G$ and $E_{\calF\cap\calG} G$. Then, the $G$-$CW$-complex $X$ given by the cellular homotopy $G$-push-out
\[
\xymatrix{
E_{\calF\cap\calG}G \ar[r] \ar[d] & E_{\calF}G \ar[d]\\
E_{\calG}G \ar[r] & X
}
\]
is a model for $E_{\calF\cup\calG}G$.
\end{lemma}
With the notation \cref{lemma:union:families} we have the following
\begin{corollary}\label{upper:bound:gd:two:families}
  $\gd_{\calG\cup\calF}(G)\leq\max\{\gd_{\calF}(G), \gd_{\calG}(G), \gd_{\calG\cap\calF}(G)+1\}.$  
\end{corollary}
\subsection*{Nested families}

Given a group $G$ and two nested families $\calF\subseteq \calG$ of $G$, we will use the following propositions to bound the geometric dimension $\gd_{\calF}(G)$ using the geometric dimension $\gd_{\calG}(G)$.

\begin{proposition}\cite[Proposition 5.1 (i)]{LW12}\label{nested:families:lw}
 Let $G$ be a group and let $\calF$ and $\calG$ be
two families of subgroups such that $\calF\subseteq \calG$. Suppose for every $H\in \calG$ we have $\gd_{\calF\cap H}(H)\leq d$. Then $\gd_{\calF}(G)\leq \gd_{\calG}(G)+d.$
\end{proposition}
The proof of the following proposition is implicit in \cite[Proof~of~theorem~3.1]{Lu00} and \cite[Proposition~5.1]{LW12}.
\begin{proposition}\label{prop:haefliger}
Let $G$ be a group. Let $\calF$ and $\calG$ be families of subgroups of $G$ such that $\calF\subseteq \calG$. If $X$ is a model for $E_\calG G$, then
\[\gd_{\calF}(G)\leq \max\{\gd_{\calF\cap G_\sigma}(G_{\sigma})+\dim(\sigma)| \ 
\sigma \text{ is a cell of } \ X \}.\] 
\end{proposition}

\subsection*{Bredon cohomology}
In this subsection, we recall the definition of Bredon cohomology, the cohomological dimension for families and its connection with the geometric dimension for families. For further details see \cite{MV03}. 

Fix a group $G$ and $\calF$ a family of subgroups of $G$. The \emph{orbit category} $\mathcal{O}_{\calF}G$ is the category whose objects are $G$-homogeneous spaces $G/H$ with $H\in \calF$ and morphisms are $G$-functions. The \emph{category of Bredon modules}  is the category whose objects are contravariant functors $M\colon \mathcal{O}_{\calF}G \to Ab$ from the orbit category to the category of  abelian groups, and morphisms are natural transformations $f\colon M\to N$. This is an abelian category  with enough projectives. The constant Bredon module  $\underline{\mathbb{Z}}\colon \mathcal{O}_{\calF}G \to Ab$ is defined in objects by $\underline{\mathbb{Z}}(G/H)=\mathbb{Z}$ and in morphisms by $\underline{\mathbb{Z}}(\varphi)=id_{\mathbb{Z}}$. 
Let $P_{\bullet}$ be a projective resolution of the  Bredon module $\undZ$, and $M$ be a Bredon module.  We define the Bredon cohomology  of $G$ with coefficients in $M$ as 
\[H_{\calF}^{*}(G;M)=H_{*}(\mor (P_{\bullet},M)).\]

We define the \emph{$\calF$-cohomological dimension} of $G$ as 
\[\cd_{\calF}(G)=\max\{n\in \mathbb{N}| \text{ there is a Bredon module }M, H_{\calF}^{*}(G;M)\neq 0 \}.\]

We have the following Eilenberg-Ganea type theorem that relates the $\calF$-cohomological dimension and the $\calF$-geometric dimension.

\begin{theorem}\cite[Theorem 0.1]{LM00}\label{geometric:cohomological:dimension} Let $G$ be a group and $\calF$ be  a family of subgroups of $G$. Then  \[\cd_{\calF}(G)\leq \gd_{\calF}(G)\leq \max\{\cd_{\calF}(G), 3\}.\]
\end{theorem} 
This \cref{geometric:cohomological:dimension} together with the following Mayer-Vietoris sequence  will be used to give lower bounds for the $\calF$-geometric dimension $\gd_{\calF}(G)$.
\subsection*{Mayer-Vietoris sequence}
Following the notation of \cref{LW}, by \cite[Proposition 7.1]{MR3252961} \cite{LW12} we have the next long exact sequence 
\begin{equation*}
\begin{split}  
\cdots\to H^{n}( X/G)\to \left(\prod_{L\in I}H^{n}(E_{ (\overline{\calF_{n+1}\cap H})[L]}N_{G}[L]/N_G[H])\right)\oplus H^n( E_{\overline{\calF_n\cap H}}G/G)\to  & \\
 \prod_{L\in I}H^{n}( E_{(\overline{\calF_n\cap H})\cap N_{G}[L]}N_{G}[L]/N_G[L])\to H^{n+1}(X/G)\to \cdots& 
  \end{split}  
\end{equation*} 
\begin{remark}\label{cohomological:version}
The results presented in \cref{lw:gd:upper:bound}, \cref{upper:bound:gd:two:families}, and \cref{nested:families:lw} have cohomological counterparts. Specifically, if we replace $\gd_{\calF}$ with $\cd_{\calF}$, all the results hold true, see for instance \cite[Remark 2.9]{tomasz}. 
\end{remark}

\section{The $\calF_k$-dimension of a virtually $\Z^n$ group}\label{sec:3}
The objective of this section is to prove \cref{Virtually:dimension:Z}. Let $G$ be a virtually $\Z^n$ group. By \cite[Proposition 1.3]{tomasz}, \cref{geometric:cohomological:dimension} and since the $\calF$-cohomological dimension is monotone, we have for all $0\leq k<n$ the following inequalities $$n+k\geq \gd_{\calF_k}(G)\geq \cd_{\calF_k}(G) \geq \cd_{\calF_k\cap\Z^n}(\Z^n).$$ Therefore, to prove \cref{Virtually:dimension:Z}, it is enough to show that  $\cd_{\calF_k\cap\Z^n}(\Z^n)\geq n+k$ for $0\leq k< n$.  In \cref{cd:lower:bound}, we prove this inequality. In order to prove  \cref{cd:lower:bound} we need  \cref{upper:bound:geometry:dimension},   Mayer-Vietoris sequence, \cref{inclusion}, and \cref{property:max:2}.

\begin{lemma}\label{upper:bound:geometry:dimension}
Let $k,t,n \in \mathbb{N}$ such that $0\leq k< t\leq n$. Let $H$ be a subgroup of $\Z^n$ of $\rank$ $t$, then $\gd_{ \calF_k\cap H}(\Z^n)\leq n+k$.
\end{lemma}
\begin{proof} 
The proof is by induction on $k$. Let $G=\Z^n$.
For $k=0$ we have 
    $\gd_{ \calF_0\cap H}(G)=\gd_{}(G)=n.$ 
    Suppose that the inequality is true for all $k<m$. We prove that the inequality is true for $k=m$. Let $\sim$ be the equivalence relation on $\calF_m\cap H-\calF_{m-1}\cap H$ defined by commensurability, and let $I$ a complete set of representatives classes in $(\calF_m\cap H-\calF_{m-1}\cap H)/\sim$. By \cref{lw:gd:upper:bound} and \cref{families} we have 
\[\gd_{\calF_m\cap H}(G)\leq\max\{\gd_{ \calF_{m-1}\cap H}(G)+1, \gd_{ (\calF_m\cap H)[L]}(G)| L\in I\}\leq\max\{n+m, \gd_{ (\calF_m\cap H)[L]}(G)| L\in I\}\]
then to prove that $\gd_{ \calF_m\cap H}(G)\leq n+m$ it is enough to prove that $\gd_{(\calF_m\cap H)[L]}(G)\leq n+m$ for all $L\in I$. Let $L\in I$.  We can write the family  $$(\calF_m\cap H)[L]=\{ K\leq G|K\in  \calF_m\cap H- \calF_{m-1}\cap H, K\sim L\}\cup (\calF_{m-1}\cap H)$$  as the union of two families $ (\calF_m\cap H)[L]=\calG \cup (\calF_{m-1}\cap H)$ where $\calG$ is the family generated by $\{ K\leq G|K\in  \calF_m\cap H- \calF_{m-1}\cap H, [K]=[L]\}$. By \cref{upper:bound:gd:two:families} we have          \[
\begin{split}
\gd_{ (\calF_m\cap H)[L]}(G) &\leq\max\{\gd_{ \calF_{m-1}\cap H}(G), \gd_{\calG \cap ( \calF_{m-1}\cap H)}(G)+1, \gd_{\calG}(G)\}\\
              &\leq\max\{n+m-1, \gd_{\calG \cap ( \calF_{m-1}\cap H)}(G)+1, \gd_{\calG}(G)\},\text{by induction hypothesis}.\end{split}
              \] 
We prove the following inequalities 
\begin{enumerate}[i)]
    \item $\gd_{\calG}(G)\leq n-m$,
    \item $\gd_{\calG \cap (\calF_{m-1}\cap H)}(G)\leq n+m-1 $
\end{enumerate}
and as a consequence we will have $\gd_{\calF_m\cap H[L]}(G)\leq n+m$. First, we prove item $i)$. Note that a model for $E_{\calF_0}(G/L)$ is a model for $E_\calG G$ via the action given by the  projection $G\to G/L$. Since   $ G/L$ is virtually $\Z^{n-m}$  by \cite[Proposition 1.3]{tomasz} we have $\gd_{\calF_0}(G/L)\leq n-m$.

Now we prove item $ii)$. Applying  \cref{nested:families:lw} to the inclusion of families $\calG\cap ( \calF_{m-1}\cap H) \subset \calG$ we get
\[\gd_{\calG \cap (\calF_{m-1}\cap H)}(G)\leq \gd_\calG(G)+d\]
for some $d$ such that for any $K\in \calG$ we have $\gd_{\calG\cap (\calF_{m-1}\cap H)\cap K}(K)\leq d$.
Since we already proved $\gd_{\calG}(G)\leq n-m$, our next task is to show that $d$ can be chosen to be equal to $2m-1$.

Recall that any $K\in \calG$ is virtually $\Z^t$ for some $0\leq t \leq m$. We split our proof into two cases. First assume that $K\in \calG$ is virtually $\Z^t$ for some $0\leq t \leq m-1$. Hence $K$ belongs to $ \calF_{m-1}\cap H$, it follows that $K$ belongs to $\calG \cap ( \calF_{m-1}\cap H)$ and we conclude $\gd_{(\calG \cap (\calF_{m-1}\cap H))\cap K}(K)=0$. Now assume $K\in \calG$ is virtually $\Z^{m}$. We claim that $(\calG \cap ( \calF_{m-1}\cap H))\cap K= \calF_{m-1}\cap K$. The inclusion $(\calG \cap (\calF_{m-1}\cap H))\cap K \subset \calF_{m-1}\cap K$ is clear since  $ \calF_{m-1}\cap H\subset\calF_{m-1}$. For the other inclusion let $M\in \calF_{m-1}\cap K$. Since $K\leq H$ we get $\calF_{m-1}\cap K \subseteq \calF_{m-1}\cap H$ and as a consequence $M\in \calF_{m-1}\cap H$,  on the other hand $M\leq K\in \calG$, therefore $M\in (\calG \cap ( \calF_{m-1}\cap H))\cap K$. This establishes the claim. We conclude that \[\gd_{(\calG \cap ( \calF_{m-1}\cap H))\cap K}(K)=\gd_{\calF_{m-1}\cap  K}(K)\leq m+m-1=2m-1\] where the inequality follows from \cite[Proposition 1.3]{tomasz}. 
\end{proof}

The following proposition is a mild generalization of \cite[Lemma 2.3]{victor:nucinkis:corob}.
\begin{proposition}
    Let $H$ be a subgroup of $\Z^n$  that is maximal in $\calF_t-\calF_{t-1}$. Then, for all $0\le k\le t$, each $L\in (\calF_k\cap H- \calF_{k-1}\cap H)$ is contained in a unique maximal element $M\in  (\calF_k- \calF_{k-1})$ and $M$ is a subgroup of $H$.
\end{proposition}
\begin{proof}
We have two cases $\rank(H)=n$ or $\rank(H)<n$. In the first case, by the maximality of $H$ we have that $H=\Z^n$ and $\calF_{k}\cap H=\calF_k$. Let $L\in (\calF_k- \calF_{k-1})$, we  consider the following short exact sequence:
\[1 \to L \to \mathbb{Z}^n \xrightarrow[]{p} \mathbb{Z}^n / L \to 1. \]
Since \(\operatorname{rank}(\mathbb{Z}^n) = \operatorname{rank}(L) + \operatorname{rank}(\mathbb{Z}^n / L)\) and by the classification theorem of finitely generated abelian groups, we have that \(\mathbb{Z}^n / L\) is isomorphic to \(\mathbb{Z}^{n-k} \oplus F\) where \(F\) is the torsion part. Therefore, it is clear that \(p^{-1}(F)\) is the unique maximal subgroup of \(\mathbb{Z}^n\) of rank \(k\) that contains \(L\).

Suppose that $\rank(H)=t<n$.  Let $L\in \calF_{k}\cap H-  \calF_{k-1}\cap H$, in particular $L\in \calF_k$ then by the first case  $L$ is contained in a unique maximal $M\in \calF_{k}- \calF_{k-1}$. We claim that $M\leq H$. Note that $MH$ is virtually $\Z^t$ because $$[HM:H]=[M:M\cap H]\le [M:L]<\infty,$$ 
it follows that  $MH\in \calF_t$, and then the maximality of $H$ implies   $H=MH$. This finishes the proof of claim.
Now it is easy to see that  $M\in  \calF_{k}\cap H- \calF_{k-1}\cap H$ is the unique maximal in  $\calF_{k}\cap H-  \calF_{k-1}\cap H$  containing $L$.  In fact, suppose that there is another  $N\in \calF_{k}\cap H-  \calF_{k-1}\cap H$ that is maximal and contains $L$. Then we have
$$[NM:N]=[M:M\cap N]\le [M:L]<\infty,$$
which implies  $NM\in \calF_k\cap H$. This contradicts the maximality of $N$.  
\end{proof}

\begin{corollary}\label{property:max:2}
    Let $H$ be a subgroup of $\Z^n$  that is maximal in $\calF_t-\calF_{t-1}$. Then, for all $0\le k\le t$ the following statements  hold
    \begin{enumerate}[a)]
        \item Each $L\in (\calF_k\cap H- \calF_{k-1}\cap H)$ is contained in a unique maximal element $M\in  (\calF_k\cap H- \calF_{k-1}\cap H)$.
        
         \item Let   $S\in (\calF_{k}\cap H-  \calF_{k-1}\cap H)$ be a maximal element, then $S$ is maximal in $\calF_k-\calF_{k-1}$.
         \end{enumerate}
      \end{corollary}

\begin{lemma}\label{lem:aux:gd:sub(L)}
Let $n, t\in \N$ such that $0\leq t<n$. Let $L$ be a subgroup of $\Z^n$ that is maximal in $\calF_t-\calF_{t-1}$. Let $\sub(L)$ be the family of all the subgroups of $L$. Then $\gd_{\sub(L)}(\Z^n)\leq n-t$.
\end{lemma}
\begin{proof}
A model for $E_{\calF_0}(\Z^n/L)$ is a model for $E_{\sub(L)}\Z^n$  via the action given by the  projection $\Z^n\to \Z^n/L$. Since $\Z^n/L=\Z^{n-t}$, a  model for $E_{\calF_0}(\Z^n/L)$ is $\mathbb{R}^{n-t}$ with the action given by translation.
\end{proof}
\begin{lemma}\label{inclusion}
      Let $p,t,n\in \mathbb{N}$ such that $0\leq k\leq p< t\leq n$. Let $H$ be a subgroup of $\Z^n$  that is maximal in $\calF_t-\calF_{t-1}$, and let $S$ be maximal in  $\calF_{p}\cap H-  \calF_{p-1}\cap H$ (note that $S$ is a subgroup of $H$).  Then, we  can choose a  model  $X$  of $E_{\calF_k\cap S}\Z^n$ with $\dim(X)\leq n+k$, and a  model $Y$ of $E_{\calF_k\cap H}\Z^n$ with $\dim(Y)\leq n+k$  such that we have an inclusion $X\xhookrightarrow{} Y$. 
\end{lemma} 

\begin{proof}
The proof is by induction on $k$. Let  $G=\Z^n$. For $k=0$  we have $E_{\calF_0\cap S}G=EG$ and $E_{\calF_0\cap H}G=EG$. A model for $EG$ is $\mathbb{R}^n$ and the claim follows. Assuming the claim holds for all $k < m$, we prove that it holds for $k = m$, i.e. we  show that there is  a  model  $X$  of $E_{\calF_m\cap S}G$ with $\dim(X)\leq n+m$, and a  model $Y$ of $E_{\calF_m\cap H}G$ with $\dim(Y)\leq n+m$  such that we have a inclusion $X\xhookrightarrow{} Y$.  Let $\sim$ be the equivalence relation on $\calF_m\cap H-\calF_{m-1}\cap H$ defined by commensurability. Let $I_1$ be a complete set of representatives of  classes  of subgroups in $ (\calF_m\cap H -\calF_{m-1}\cap H)/\sim$. By \cref{property:max:2} these representatives can be chosen to be maximal within their class. Applying \cref{LW} and \cref{families}, the following homotopy $G$-push-out  gives us a model $X_1$ for $E_{\calF_m\cap H}G$ 
\begin{equation}\label{push-out:a:1}
\begin{tikzpicture}
  \matrix (m) [matrix of math nodes,row sep=3em,column sep=4em,minimum width=2em]
  {
     \displaystyle\bigsqcup_{L\in I_1} E_{\calF_{m-1}\cap H}G & E_{\calF_{m-1}\cap H}G \\
      \displaystyle\bigsqcup_{L\in I_1} E_{ (\calF_m\cap H)[L] }G & X_1 \\};
  \path[-stealth]
    (m-1-1) edge node [left] {$\displaystyle\bigsqcup_{L\in I_1}f_{L}$} (m-2-1) (m-1-1.east|-m-1-2) edge  node [above] {$\displaystyle\bigsqcup_{L\in I_1} id_{}$} (m-1-2)
    (m-2-1.east|-m-2-2) edge node [below] {} (m-2-2)
    (m-1-2) edge node [right] {} (m-2-2);
\end{tikzpicture}
\end{equation}
For $L\in I_1$, by maximality of $L$ in its commensuration class we can write the family 
$$(\calF_m\cap H)[L]=\{ K\leq G|K\in  \calF_m\cap H- \calF_{m-1}\cap H, K\sim L\}\cup (\calF_{m-1}\cap H)$$
as the union of two families 
$$(\calF_{m}\cap H)[L]=\sub(L)\cup (\calF_{m-1}\cap H),$$
where $\sub(L)$ is the family of all the subgroups of $L$.

On the other hand. Let $\sim$ be the equivalence relation on $\calF_m\cap S-\calF_{m-1}\cap S$ defined by commensurability. Let $I_2$ be a complete set of representatives of classes  of subgroups in $ (\calF_m\cap S -\calF_{m-1}\cap S)/\sim$. By \cref{property:max:2}, these representatives can be chosen to be maximal  within their class. Applying \cref{LW}, we obtain a homotopy $G$-push-out that gives us a model $X_2$ for $E_{\calF_m\cap S}G$
\begin{equation}\label{push-out:a:3}
\begin{tikzpicture}
  \matrix (m) [matrix of math nodes,row sep=3em,column sep=4em,minimum width=2em]
  {
\displaystyle\bigsqcup_{L\in I_2} E_{ \calF_{m-1}\cap S}G & E_{\calF_{m-1}\cap S}G \\
      \displaystyle\bigsqcup_{L\in I_2} E_{ \sub(L)\cup (\calF_{m-1}\cap S) }G & X_2 \\};
  \path[-stealth]
    (m-1-1) edge node [left] {} (m-2-1) (m-1-1.east|-m-1-2) edge  node [above] {} (m-1-2)
    (m-2-1.east|-m-2-2) edge node [below] {} (m-2-2)
    (m-1-2) edge node [right] {} (m-2-2);
\end{tikzpicture}
\end{equation}

Let $T \in I_2$. We claim that a model for $E_{\sub(T) \cup (\calF_{m-1} \cap H)}G$ is also a model for $E_{\sub(L) \cup (\calF_{m-1} \cap H)}G$ for every $L \in I_1$. Let $L \in I_1$. Note that $T$ and $L$ are maximal subgroups of $H$, thus $H = L \oplus N_1$ and $H = T \oplus N_2$. We can construct an automorphism of $H$, $\sigma\colon L \oplus N_1 \to T \oplus N_2$, that maps $L$ to $T$ isomorphically. Since $H$ is maximal in $G$, we can split $G$ as $G = H \oplus R$. Therefore, we can extend the automorphism $\sigma$ to an automorphism of $G$, $\Hat{\sigma}\colon L \oplus N_1 \oplus R \to T \oplus N_2 \oplus R$, that maps $L$ to $T$ isomorphically and preserves the subgroup $H$. It follows that $E_{\sub(T)\cup (\calF_{m-1}\cap H)}G$ is a model for $E_{\sub(L)\cup (\calF_{m-1}\cap H)}G$ via the action given by the automorphism $\Hat{\sigma}$. From \cref{property:max:2} it follows that  $I_1=I_2\sqcup (I_1-I_2)$. Therefore, we can replace the homotopy $G$-push-outs in \cref{push-out:a:1} and \cref{push-out:a:3} with the following homotopy $G$-push-outs.

\begin{equation}\label{push-out:a:4}
\begin{tikzpicture}
  \matrix (m) [matrix of math nodes,row sep=3em,column sep=4em,minimum width=2em]
  {
    \left(\displaystyle\bigsqcup_{L\in I_2}E_{\calF_{m-1}\cap H}G\right) \bigsqcup \left(\displaystyle\bigsqcup_{L\in I_1-I_2}  E_{\calF_{m-1}\cap H}G\right) & E_{\calF_{m-1}\cap H}G \\
    \left(\displaystyle\bigsqcup_{L\in I_2}  E_{ \sub(T)\cup ( \calF_{m-1}\cap H) }G\right) \bigsqcup \left(\displaystyle\bigsqcup_{L\in I_1-I_2}  E_{ \sub(T)\cup (\calF_{m-1}\cap H) }G \right)& X_1 \\};
  \path[-stealth]
    (m-1-1) edge node [left] {$\displaystyle\bigsqcup_{L\in I_2\sqcup (I_1-I_2)}f_{T}$} (m-2-1) (m-1-1.east|-m-1-2) edge  node [above] {$\displaystyle\bigsqcup_{L\in I_2\sqcup (I_1-I_2)} id_{}$} (m-1-2)
    (m-2-1.east|-m-2-2) edge node [below] {} (m-2-2)
    (m-1-2) edge node [right] {} (m-2-2);
\end{tikzpicture}
\end{equation}
\begin{equation}\label{push-out:a:3:2}
\begin{tikzpicture}
  \matrix (m) [matrix of math nodes,row sep=3em,column sep=4em,minimum width=2em]
  {
\displaystyle\bigsqcup_{L\in I_2} E_{ \calF_{m-1}\cap S}G & E_{\calF_{m-1}\cap S}G \\
      \displaystyle\bigsqcup_{L\in I_2} E_{ \sub(T)\cup (\calF_{m-1}\cap S) }G & X_2 \\};
  \path[-stealth]
    (m-1-1) edge node [left] {} (m-2-1) (m-1-1.east|-m-1-2) edge  node [above] {} (m-1-2)
    (m-2-1.east|-m-2-2) edge node [below] {} (m-2-2)
    (m-1-2) edge node [right] {} (m-2-2);
\end{tikzpicture}
\end{equation}

By induction hypothesis there is a model  $X$ of $E_{\calF_{m-1}\cap S}G$ with $\dim(X)\leq n+m-1$, and a  model $Y$ of $E_{\calF_{m-1}\cap H}G$ with $\dim(Y)\leq n+m-1$,  such that we have a inclusion $X\xhookrightarrow{} Y$. By the  $G$-push-outs in  \cref{push-out:a:4} and \cref{push-out:a:3:2},  to prove that there is a inclusion $E_{ \calF_{m}\cap S}\xhookrightarrow{}E_{ \calF_{m}\cap H}$ it is enough to prove that there is a inclusion $E_{\sub(T)\cup (\calF_{m-1}\cap S) }G \xhookrightarrow{}E_{\sub(T)\cup (\calF_{m-1}\cap H)}G$. By \cref{lemma:union:families} the following $G$-push-outs gives us a model for  $E_{\sub(T)\cup (\calF_{m-1}\cap S) }G$ and $E_{\sub(T)\cup (\calF_{m-1}\cap H) }G$ respectively.
\begin{equation}\label{union:families:inlcu}
\xymatrix{
 E_{\sub(T)\cap (\calF_{m-1}\cap S) }G \ar[r] \ar[d] & E_{\calF_{m-1}\cap S }G\ar[d]\\
E_{\sub(T)}G\ar[r] & Y_2
} \quad 
\xymatrix{
 E_{\sub(T)\cap (\calF_{m-1}\cap H) }G \ar[r] \ar[d] & E_{\calF_{m-1}\cap H }G\ar[d]\\
E_{\sub(T)}G\ar[r] & Y_1
}
\end{equation}
Note that $\sub(T)\cap (\calF_{m-1}\cap S)=\calF_{m-1}\cap T= \sub(T)\cap (\calF_{m-1}\cap H)$. It follows from these $G$-push-outs  that we have a inclusion $E_{\sub(T)\cup (\calF_{m-1}\cap S) }G \xhookrightarrow{}E_{\sub(T)\cup (\calF_{m-1}\cap H)}G$.

Finally, we prove that  $\dim(X_1)\leq n+m$ and $\dim(X_2)\leq n+m$. From \cref{push-out:a:4} it follows 
\[
\begin{split}
\dim(X_1) &\leq\max\{\gd_{ \calF_{m-1}\cap H}(G), \gd_{ \calF_{m-1}\cap H}(G)+1, \gd_{ \sub(T)\cup (\calF_{m-1}\cap H)}(G)\}\\
&\leq\max\{n+m, \gd_{ \sub(T)\cup (\calF_{m-1}\cap H)}(G)\}\text{, by induction hypothesis}
\end{split}\]
Then to prove that $\dim(X_1)\leq n+m$  it is enough to prove $\gd_{\sub(T)\cup(\calF_{m-1}\cap H)}(G)\leq n+m$.  By \cref{union:families:inlcu} and  since $\sub(T) \cap ( \calF_{m-1}\cap H)=\calF_{m-1}\cap T$ we have 
\[
\begin{split}
\gd_{ \sub(T)\cup(\calF_{m-1}\cap H)}(G) &\leq \dim(Y_1)\\
              &\leq\max\{\gd_{ \calF_{m-1}\cap H}(G), \gd_{\sub(T) \cap ( \calF_{m-1}\cap H)}(G)+1, \gd_{\sub(T)}(G)\}\\
&=\max\{\gd_{ \calF_{m-1}\cap H}(G), \gd_{\calF_{m-1}\cap T}(G)+1, \gd_{\sub(T)}(G)\}\\
              &\leq \max\{n+m-1, n+m, n-m\}\text{, By \cref{upper:bound:geometry:dimension} and \cref{lem:aux:gd:sub(L)}}\\
              &=n+m.
              \end{split}
              \] 
              \end{proof}

\begin{theorem}[The lower bound]\label{cd:lower:bound}
Let $m,t,n\in \mathbb{N}$ such that $0\leq m< t\leq n$. Let $H$ be a subgroup of $\Z^n$ that is maximal in $\calF_t-\calF_{t-1}$, then  $H_{ \calF_m\cap H}^{n+m}(\Z^n;\undZ)\neq 0$.
\end{theorem}
\begin{proof}
Let $G=\Z^n$. The proof is by double induction on $(t,m)$. 
The claim is true  for all $(t,0)\in \mathbb{N}\times \{0\}$.  
Let $H$ be a subgroup of $G$ that is maximal in $\calF_t-\calF_{t-1}$, then $$H_{ \calF_0\cap H}^{n+0}(G;\undZ)=H_{\calF_0}^{n}(G;\undZ)=H^{n}(G;\Z)=\Z.$$ 
Suppose that  the claim is true for all $(t,s)\in \mathbb{N}\times \{0,1,\dots, m-1\}$, we prove that the claim is true for $(t,m)$, i.e. $H_{ \calF_m\cap H}^{n+m}(G;\undZ)\neq 0$.
 
Applying  Mayer-Vietoris to the $G$-push-out in \cref{push-out:a:4} and \cref{upper:bound:geometry:dimension}, we have the following long exact sequence
\begin{equation}\label{mv1}
    \begin{split}
        \cdots\to \left(\prod_{L\in \I_1}H^{n+m-1}(E_{\sub(T)\cup(\calF_{m-1}\cap H)}G/G)\right)\oplus H^{n+m-1}(E_{\calF_{m-1}\cap H}G/G)\xrightarrow[]{\varphi}& \\
        \prod_{L\in I_1}H^{n+m-1}(E_{\calF_{m-1}\cap H}G/G)\to H^{n+m}(X_1/G) \to \prod_{L\in \I_1}H^{n+m}(E_{\sub(T)\cup(\calF_{m-1}\cap H)}G/G)& \to 0    \end{split}
\end{equation}

We now show that $\prod_{L\in \I_1}H^{n+m}(E_{\sub(T)\cup(\calF_{m-1}\cap H)}G/G)=0$. It is enough to show that  $\gd_{\sub(T)\cup (\calF_{m-1}\cap H)}(G)\leq n+m-1$. 
By  \cref{lemma:union:families}  the following homotopy  $G$-push-out gives us a model $Y$ for $E_{\sub(T)\cup (\calF_{m-1}\cap H)}G$.

\begin{equation}\label{push-out:2}
\begin{tikzpicture}
  \matrix (m) [matrix of math nodes,row sep=3em,column sep=4em,minimum width=2em]
  {
      E_{\sub(T)\cap (\calF_{m-1}\cap H)}G  & E_{\calF_{m-1}\cap H}G\\
      E_{\sub(T)}G& Y\\};
  \path[-stealth]
    (m-1-1) edge node [left] {} (m-2-1) (m-1-1.east|-m-1-2) edge  node [above] {$g$} (m-1-2)
    (m-2-1.east|-m-2-2) edge node [below] {$h$} (m-2-2)
    (m-1-2) edge node [right] {$\psi$} (m-2-2);
\end{tikzpicture}
\end{equation}
Note that $\sub(T)\cap(\calF_{m-1}\cap H)= \calF_{m-1}\cap T$. By \cref{inclusion}, the map $g$ can be taken as an inclusion, then by \cite[Theorem 1.1]{waner} the  homotopy  $G$-push-out can be taken as a $G$-push-out.  It follows that 
\begin{equation}\label{bound:union:two:families}
\begin{split}
    \gd_{\sub(T)\cup (\calF_{m-1}\cap H)}(G)& \leq \dim(Y)\\
                      &=\max \{\gd_{\sub(T)}(G), \gd_{\calF_{m-1}\cap T}(G), \gd_{\calF_{m-1}\cap H}(G)\}\\
                      &\leq\max\{n-m,n+m-1,n+m-1\}\text{, by \cref{upper:bound:geometry:dimension} and \cref{lem:aux:gd:sub(L)}}\\
                      &=n+m-1
    \end{split}
\end{equation}

Then the sequence \cref{mv1} reduce to 

\begin{equation*}\label{mv2}
    \begin{split}
        \cdots\to \left(\prod_{L\in \I_1}H^{n+m-1}(E_{\sub(T)\cup(\calF_{m-1}\cap H)}G/G)\right)\oplus H^{n+m-1}(E_{\calF_{m-1}\cap H}G/G)\xrightarrow[]{\varphi}& \\
        \prod_{L\in I_1}H^{n+m-1}(E_{\calF_{m-1}\cap H}G/G)\to H^{n+m}(X_1/G) \to &  0    \end{split}
\end{equation*}

Then to prove that $H_{\calF_k\cap H}^{n+m}(G;\undZ)=H^{n+m}(X_1/G)\neq 0$ is enough  to prove that $\varphi$ is not surjective. By \cref{push-out:a:4}  we have $\varphi= (\prod_{L\in I_1}f_{T}^{*})-\Delta$, where $\Delta$ is the diagonal embedding. First, we prove that $f_{T}^{*}$ is not surjective.

Applying  Mayer-Vietoris to the $G$-push-out in \cref{push-out:2} we have the following long exact sequence 
\begin{equation*}
    \begin{split}
      \cdots \to H^{n+m-1}(E_{\sub(T)\cup (\calF_{m-1}\cap H)}G/G)\xrightarrow[]{h^{*}\oplus \psi^{*}}   H^{n+m-1}(E_{\sub(T)}G/G)\oplus H^{n+m-1}(E_{\calF_{m-1}\cap H}G/G)\to &\\
       H^{n+m-1}(E_{ \calF_{m-1}\cap T}G/G)\to 0 &\\
    \end{split}
\end{equation*}
Since $\gd_{\sub(T)}(G)\leq n-m$ and since there is precisely one $G$-map  $E_{\calF_{m-1}\cap H}G\to E_{\sub(T)\cup (\calF_{m-1}\cap H)}G$ up to $G$-homotopy we can reduce the  sequence to 
\begin{equation*}
    \begin{split}
      \cdots \to H^{n+m-1}(E_{\sub(T)\cup (\calF_{m-1}\cap H)}G/G)\xrightarrow[]{f_T^{*}}   H^{n+m-1}(E_{\calF_{m-1}\cap H}G/G)\to &\\
       H^{n+m-1}(E_{ \calF_{m-1}\cap T}G/G)\to 0 &\\
    \end{split}
\end{equation*}
By hypothesis $T$ is maximal in $ \calF_m\cap H-\calF_{m-1}\cap H$, then by \cref{property:max:2} $b)$ we have that $T$ is maximal in $\calF_m- \calF_{m-1}$,  by induction hypothesis we have that $H^{n+m-1}(E_{ \calF_{m-1}\cap T}G/G)\neq 0$, thus $f_T^{*}$ is not surjective. 

Finally, we see that $\varphi$ is not surjective. In fact, let $b_K\notin \Ima(f_{T}^{*})$, for some $K\in I_1$, then $(0,0, \cdots, b_K,\cdots,0)\notin \Ima(\varphi)$. Suppose that is not the case, i.e. there is  $$(\prod_{L\in \I_1}a_L,c)\in \left(\prod_{L\in \I_1}H^{n+m-1}(E_{\sub(T)\cup(\calF_{m-1}\cap H)}G/G)\right)\oplus H^{n+m-1}(E_{\calF_{m-1}\cap H}G/G)$$ such that  $(0,0, \cdots, b_K,\cdots,0)=\varphi((\prod_{L\in \I_1}a_L,c))=\prod_{L\in \I_1}f_{T}^{*}(a_L)-\Delta(c)=(f_{T}^{*}(a_L)-c)_{L\in I_1}.$
Then $f_{T}^{*}(a_L)=c$ for $L\neq K$ and $f_{T}^{*}(a_K)-c=b_K$, it follows that $$b_K=f_{T}^{*}(a_K)-f_{T}^{*}(a_L)=f_{T}^{*}(a_K-a_L),$$ then $b_K\in \Ima(f_{T}^{*})$ and this is a contradiction.
\end{proof}

\begin{proposition}\label{fk:dimension:Z}
Let $k,t,n\in \mathbb{N}$ such that $0\leq k< t\leq n$. Let $H$ be a subgroup of $\Z^n$ that is maximal in $\calF_t-\calF_{t-1}$. Let $\calF_k\cap H$ be the family that consists of all the subgroups of $H$ that belong to $\calF_k$. Then $\cd_{\calF_k\cap H}(\Z^n)=\gd_{\calF_k\cap H}(\Z^n)=n+k$.
\end{proposition}

\section{Some applications of \cref{subgroup:Z:lower:bound:1}}\label{sec:4}

\subsection{The $\calF_k$-dimension of braid groups }

In this subsection, we compute the $\calF_n$-dimension of full and pure braid groups. For our purposes, it is convenient to define the braid group as follows: let $D_n$ the closed disc with $n$ punctures,  we  define the {\it braid group $B_n$ on $n$ strands},  as the  isotopy classes of orientation preserving diffeomorphisms of $D_n$ that restrict to the identity on the boundary $\partial D_n$. We define the {\it pure braid group}, $P_n$, as the finite index subgroup of $B_n$ consisting of elements that  fixe point-wise the punctures.

\begin{theorem} \label{fk:dimension:Bn}
    Let  $k,n \in \N$ such that $0\leq k< n-1$ and let $G$ be either the braid group  $B_n$ or the pure braid group $P_n$. Then  $\gd_{\calF_k}(G)=\cd_{\calF_k}(G)=n+k-1$.
\end{theorem}
\begin{proof}
It is enough to prove the following inequalities $$n+k-1\geq  \gd_{\calF_k}(G)\geq \cd_{\calF_k}(G)\geq n+k-1.$$
In \cite[Theorem 1.4]{Rita:Porfirio:Luis} was proved that $\gd_{\calF_k}(B_n)\leq \vcd(B_n)+k$ for all $0\leq k< n-1$. Since $P_n$ has finite index in $B_n$ also we have $\gd_{\calF_k}(P_n)\leq \vcd(P_n)+k$ for all $0\leq k< n-1$. On the other hand, it is well known that  $\vcd(B_n)=n-1$ see for example \cite[Section 3]{Arnold2014}. This proves the first inequality. The second inequality is by \cref{geometric:cohomological:dimension}.
In  \cite[Proposition 3.7]{Ramon:Juan}  it is shown that   $P_n$ has a subgroup isomorphic to $\Z^{n-1}$. Therefore, by monotonicity of the $\calF_k$-geometric dimension and \cref{subgroup:Z:lower:bound:1} we have $\cd_{\calF_k}(B_n)\geq \cd_{\calF_k}(P_n)\geq n+k-1$ for all  $0\leq k< n-1$. This proves the last inequality.
\end{proof}
For $k=1$, this theorem has been proved in \cite{Ramon:Juan}.

\subsection{The $\calF_k$-dimension of RAAGs and  their outer automorphism groups}\label{sec:4.2}

In this subsection, we compute the $\calF_n$-dimension of RAAGs and we give a lower bound for the $\calF_n$-geometric dimension of the outer automorphism group  of some RAAGs.

We recall some basic notions about RAAGs\textcolor{blue}{,} for further details see for instance \cite{MR2322545}.
Let $\Gamma$ be a finite simple graph, i.e. a finite graph without loops or multiple edges between vertices. We define the {\it right-angled Artin group} (RAAG) $A_{\Gamma}$ as the group generated by the vertices of $\Gamma$ with all the relations of the form $vw=wv$ 
 whenever $v$ and $w$ are joined by an edge. 
\subsubsection*{The Salvetti complex}
For the construction of the Salvetti complex we follow \cite[Subsection 3.6]{MR2322545}.
Let $A_\Gamma$ be a RAAG, its \emph{Salvetti complex} $S_\Gamma$ is a CW-complex that can be constructed as follows:

\begin{itemize}
    \item The $S_\Gamma^{(1)}$ skeleton is constructed as follows: we take a point $x_0$, and for each $v\in V(\Gamma)$, we attach a $1$-cell $I=[0,1]$ that identifies the endpoints of $I$ to $x_0$. Then, the $S_\Gamma^{(1)}$ skeleton is a wedge of circles.

    \item The $S_\Gamma^{(2)}$ skeleton  is constructed as follows. For each edge of $\Gamma$ we attach a $2$-cell $I\times I$ to $S_\Gamma^{(1)}$ by the boundary $\partial (I\times I)$ as $s_vs_w s_{v}^{-1}s_{w}^{-1}$.

    \item  In general the $S_\Gamma^{(n)}$ skeleton  is constructed as follows. For each complete subgraph $\Gamma'$ of $\Gamma$ with $|V(\Gamma')|=n$ we attach a $n$-cell $I^n$ to the $S_\Gamma^{(n-1)}$ skeleton using the generators $V(\Gamma')$.
\end{itemize}
\begin{remark}\label{subgroup:abelian:rank:vcd:salvetti}
 Note that, by the construction of the Salvetti complex $S_\Gamma$, its fundamental group is $A_\Gamma$. Additionally, $S_\Gamma$ has a $\dim(S_\Gamma)$-dimensional torus embedded in it, which follows from its construction. Therefore, the fundamental group $\pi_1(S_\Gamma,x_0)=A_\Gamma$ has a subgroup that is isomorphic to $\Z^{\dim(S_\Gamma)}$.
\end{remark}

\begin{theorem}\cite[Theorem 3.6]{MR2322545}\label{model:EG:RAAG}
The universal cover of the Salvetti complex, $\Tilde{S_\Gamma}$, is a $\cat(0)$ cube complex. In particular, $S_\Gamma$ is a $K(A_\Gamma, 1)$ space.
\end{theorem}

\begin{corollary}
Let $G$ be a RAAG. Then $G$ is torsion-free.
\end{corollary}

\begin{lemma}\label{subgroup:RAAG:vcd} Let $A_\Gamma$ be a RAAG then $\gd(A_\Gamma)=\cd(A_\Gamma)=\dim(S_\Gamma)$. Moreover $$\cd(A_\Gamma)=\max\{n\in \mathbb{N}| \text{ there a complete subgraph $\Gamma'$ of $\Gamma$ with }  |V(\Gamma')|= n \}.$$ 
\end{lemma}
\begin{proof}
It is enough to prove the following inequalities $$\dim(S_\Gamma)\geq \gd(A_\Gamma)\geq \cd(A_\Gamma)\geq \dim(S_\Gamma).$$
The first inequality follows from \cref{model:EG:RAAG}. The second inequality follows from \cref{geometric:cohomological:dimension}. 
By \cite[Subsection 3.7]{MR2322545} $H^{\dim(S_\Gamma)}(S_\Gamma)=H^{\dim(S_\Gamma)}(A_\Gamma)$ is a free abelian generated by each $\dim(S_\Gamma)$-cell. The third inequality  follows. 

By construction of the Salvetti complex $S_\Gamma$ we have that $$\dim(S_\Gamma)=\max\{n\in \mathbb{N}| \text{ there a complete subgraph $\Gamma'$ of $\Gamma$ with }  |V(\Gamma')|= n \}.$$ 
Since $\cd(A_\Gamma)= \dim(S_\Gamma)$ the claim  follows.
\end{proof}

Let $G$ be a right-angled Artin group. In \cite[Corollary 1.2]{tomasz}, it was proved that $\cd_{\calF_k}(G)\leq\cd(G)+k+1$ for all $0\leq k<\cd(G)$. However, by following their proof in \cite[Proof of Theorem 3.1]{tomasz} and using \cite[Proposition 7.3]{HP20}, we can actually prove that $\cd_{\calF_k}(G)\leq\cd(G)+k$ for all $0\leq k<\cd(G)$. 
In \cite{tomasz} and \cite[Proposition 7.3]{HP20}, they work with the  $\calF_k$-cohomological dimension instead of $\calF_k$-geometric dimension, that  is the reason the following \cref{upper:bound:cd:RAAG} is stated in terms of $\calF_k$-cohomological dimension.

\begin{theorem}\label{upper:bound:cd:RAAG}
  Let $G$ be a RAAG. Then  $\cd_{\calF_k}(G)\leq\cd(G)+k$ for $k\in \mathbb{N}$.
\end{theorem}
\begin{proof}
The proof is by induction on $k$. For $k=0$ it follows from \cref{subgroup:RAAG:vcd}. Suppose that the inequality is true for all $k<m$. We prove the inequality for $k=m$. 
Let $\sim$ be the equivalence relation on $\calF_m-\calF_{m-1}$ defined by commensurability, and let $I$ be a complete set of representatives of conjugacy classes in $(\calF_m-\calF_{m-1})/\sim$. Then by the cohomological version of  \cref{lw:gd:upper:bound} (see \cref{cohomological:version}) we have 
$$\cd_{\calF_{m}}(G)\leq \max\{ \cd_{ \calF_{m-1}}(G)+1, \cd_{ \calF_{m}[L]}(N_{G}[L])| L\in I\}\leq\max\{\cd(G)+m,\cd_{ \calF_{m}[L]}(N_{G}[L])| L\in I \}.$$
Then to prove that $\cd_{\calF_{m}}(G)\leq \cd(G)+m$ it is enough to prove that $\cd_{ \calF_{m}[L]}(N_{G}[L])\leq \cd(G)+m$ for all $L\in I$. Let $L\in I$, we can write the family 
\[\calF_m[L]=\{ K\leq N_G[L]|K\in  \calF_m- \calF_{m-1}, K\sim L\}\cup (\calF_{m-1}\cap N_{G}[L])\]
as the union of two families $ \calF_m[L]=\calG \cup (\calF_{m-1}\cap N_{G}[L])$ where $\calG$ is the family generated by $\{ K\leq N_G[L]|K\in  \calF_m- \calF_{m-1}, K\sim L\}$. By the cohomological version of \cref{upper:bound:gd:two:families} (see \cref{cohomological:version}) we have 
\[
\begin{split}
\cd_{\calF_m[L]}(N_{G}[L])&\leq\max\{ \cd_{\calG}(N_{G}[L]), \cd_{\calF_{m-1}\cap N_{G}[L]}(N_{G}[L]), \cd_{\calG\cap \calF_{m-1}}(N_{G}[L])+1 \}\\
           &\leq\max\{ \cd_{\calG}(N_{G}[L]), \cd(G)+m-1, \cd_{\calG\cap \calF_{m-1}}(N_{G}[L]) +1\}\\
\end{split}\]

We prove that 
\begin{enumerate}
    \item $\cd_{\calG}(N_{G}[L])\leq \cd(G)-m$
    \item$\cd_{\calG\cap \calF_{m-1}}(N_{G}[L])\leq \cd(G)+m-1$
\end{enumerate}
As a consequence we will have $\cd_{\calF_m[L]}(N_{G}[L])\leq \cd(G)+m$.  First, we prove item $(1)$. We define the family $\calF=\{K\leq N_{G}[L] \mid [K:K\cap L]<\infty\}$. We claim that \(\calF = \calG\). To show that \(\calG \subseteq \calF\), note that
\[
\{ K \leq N_G[L] \mid K \in \calF_m - \calF_{m-1}, K \sim L \} \subseteq \{ K \leq N_{G}[L] \mid [K : K \cap L] < \infty \} = \calF
\]
since, by definition, \(\calG\) is the smallest family that contains \(\{ K \leq N_G[L] \mid K \in \calF_m - \calF_{m-1}, K \sim L \}\), it follows that \(\calG \subseteq \calF\). Now let's prove the other inclusion \(\calF \subseteq \calG\). Let \(S \in \calF\), then \([S : S \cap L] < \infty\). Note that \([LS : L] = [S : S \cap L] < \infty\), it follows that \(LS\) is commensurable with \(L\), and as a consequence \(S \leq LS \in \calG\), in particular it follows that \(S \in \calG\). This proves the claim.
 Since $\calG=\calF$ we have by  \cite[Proposition 7.3 and Definition 7.2]{HP20} that $\cd_{\calG}(N_{G}[L])\leq \cd(G)-m$. 
 
 We now prove the item $(2)$. Applying  the cohomological version of \cref{nested:families:lw} (see \cref{cohomological:version}) to the inclusion of families $(\calG\cap  \calF_{m-1}) \subset \calG$ we get
\[\cd_{\calG \cap \calF_{m-1}}(N_{G}[L])\leq \cd_\calG(N_{G}[L])+d\]
for some $d$ such that for any $K\in \calG$ we have $\cd_{(\calG\cap \calF_{m-1})\cap K}(K)\leq d$.
Since we already proved $\cd_{\calG}(N_{G}[L])\leq \cd(G)-m$, our next task is to show that $d$ can be chosen to be equal to $2m-1$.

Recall that any $K\in \calG$ is virtually $\Z^t$ for some $0\leq t \leq m$. We split our proof into two cases. First assume that $K\in \calG$ is virtually $\Z^t$ for some $0\leq t \leq m-1$. Hence $K$ belongs to $ \calF_{m-1}$, it follows that $K$ belongs to $\calG \cap \calF_{m-1}$ and we conclude $\cd_{\calG \cap \calF_{m-1}\cap K}(K)=0$. Now assume $K\in \calG$ is virtually $\Z^{m}$. We claim that $(\calG \cap  \calF_{m-1})\cap K= \calF_{m-1}\cap K$. The inclusion $(\calG \cap \calF_{m-1})\cap K \subset \calF_{m-1}\cap K$ is clear. For the other inclusion let $M\in \calF_{m-1}\cap K$. Since  $M\leq K\in \calG$, therefore $M\in (\calG \cap  \calF_{m-1})\cap K$. This establishes the claim. We conclude that \[\cd_{(\calG \cap  \calF_{m-1})\cap K}(K)=\cd_{\calF_{m-1}\cap  K}(K)\leq m+m-1=2m-1\] where the inequality follows from \cite[Proposition 1.3]{tomasz}. 
\end{proof}

\begin{theorem}\label{virtually:abelian:RAAG}
Let $G$ be a right-angled Artin group.  Then for $0 \leq k < \cd(G)$ we have
$ \cd_{\calF_k}(G)=\cd(G)+k$.
\end{theorem}

\begin{proof}
 By  \cref{upper:bound:cd:RAAG} we have $\cd_{\calF_k}(G)\leq\cd(G)+k$. On the other hand, by \cref{subgroup:RAAG:vcd}
$G$ has a subgroup isomorphic to $\Z^{\cd(G)}$, then the claim it follows from \cref{subgroup:Z:lower:bound:1}.
\end{proof}
\begin{theorem}\label{virtually:abelian:RAAG:gd}
Let $G$ be a right-angled Artin group.  Then for $0 \leq k < \cd(G)$ we have
$ \gd_{\calF_k}(G)= \cd_{\calF_k}(G)$.
\end{theorem}
\begin{proof}
If $k=0$ the claim  follows from \cref{subgroup:RAAG:vcd}. Suppose that $k\geq1$, hence by hypothesis, $\cd(G)\geq 2$. By \cref{virtually:abelian:RAAG} we have $\cd_{\calF_k}(G)\geq 3$, then by \cref{geometric:cohomological:dimension}, $\gd_{\calF_k}(G)= \cd_{\calF_k}(G)$.
\end{proof}

Given a fixed right-angled Artin group $A_{\Gamma}$,  we denote by $\Aut(A_\Gamma)$ the group of automorphisms of $A_\Gamma$ and  by $\Inn(A_\Gamma)$ the subgroup consisting of inner automorphisms. The outer automorphism group of  $A_\Gamma$ is defined as the quotient $\Out(A_\Gamma)=\Aut(A_\Gamma)/\Inn(A_\Gamma)$. If  $S\subseteq V(\Gamma)$ then the subgroup $H$ generated by $S$ is called a special subgroup of $A_\Gamma$. It  can be proven that, in fact, $H$ is the  right-angled Artin group $A_S$ associated with the full subgraph induced by $S$ in  $\Gamma$.

If $\Delta$ is a full subgraph of $\Gamma$, we denote by $A_{\Delta}$  the special subgroup generated by the vertices contained in $\Delta$. An outer automorphism $F$ of $A_\Gamma$ preserves $A_\Delta$ if there exists a representative $f\in F$ that restricts to an automorphism of $A_\Delta$. An outer automorphism $F$ acts trivially on $A_{\Delta}$ if there exists representative $f\in F$ that acts as the identity on $A_\Delta$. 

\begin{definition}
Let $\calG$, $\calH$ be two  collections of special subgroups of $A_\Gamma$. The relative outer automorphism group $\Out(A_{\Gamma};\calG,\calH^t)$ consists of automorphisms that preserve each $A_\Delta \in \calG$ and act trivially on each $A_\Delta\in \calH.$
\end{definition}

\begin{proposition}
    Let $A_{\Gamma}=A_{\Delta_1}*A_{\Delta_1}*\cdots*A_{\Delta_k}*F_n$ be a free factor decomposition of a right-angled Artin group with $k\geq 1$. Then $\gd_{\calF_k}(\Out(A_{\Gamma};\{A_{\Delta_i}\}^{t}))\geq  \vcd(\Out(A_{\Gamma};\{A_{\Delta_i}\}^{t}))+k$ for all $0\leq k<\vcd(\Out(A_{\Gamma};\{A_{\Delta_i}\}^{t})).$  
\end{proposition}
\begin{proof}
    By \cite[Theorem A]{MR4224528} $\Out(A_{\Gamma};\{A_{\Delta_i}\}^{t})$ has a free abelian subgroup of $\rank$ equal to  $\vcd(\Out(A_{\Gamma};\{A_{\Delta_i}\}^{t}))$. The inequality follows from \cref{subgroup:Z:lower:bound:1}.
\end{proof}
Let $F_n$ be the free group in $n$ generators. The group $F_n$ can be seen as the RAAG associated with the graph that has $n$ vertices and no edges. In \cite{MR830040} was proved that  $\vcd(\Out(F_n))=2n-3$ for $n\geq 2$ and  that $\Out(F_n)$ has a subgroup ismorphic to $\Z^{\vcd(\Out(F_n))}$. From \cref{subgroup:Z:lower:bound:1} we get the following 
\begin{proposition}
    Let $n\geq 2$. Let $F_n$ be the free group in $n$ generators. Then $\gd_{\calF_k}(\Out(F_n))\geq 2n+k-3$ for all $0\leq k< 2n-3$. 
\end{proposition}

Let $A_d$ be the right-angled Artin group given by a string of $d$ diamonds. In \cite[Proposition 6.5]{MR4072157} was proved that $\vcd(Out(A_d))=4d-1$ and $Out(A_d)$ has a subgroup  isomorphic to $\Z^{\vcd(\Out(A_d))}$, from \cref{subgroup:Z:lower:bound:1} we have 
\begin{proposition}
  Let $A_d$ be the right-angled Artin group given by a string of $d$ diamonds. Then $\gd_{\calF_k}(\Out(A_d))\geq 4d +k-1$ for all $0\leq k< 4d-1$. 
\end{proposition}

\subsection{The $\calF_k$-geometric dimension for graphs of groups of finitely generated virtually abelian groups.}\label{sec:5}
The objective of this section is to explicitly calculate the $\calF_n$-geometric dimension of the fundamental group of a graph of groups whose vertex groups are finitely generated virtually abelian groups, and whose edge groups are finite groups.
\subsubsection*{Bass-Serre theory}
 We recall some basic notions about Bass-Serre theory, for further details see \cite{Se03}. A \emph{graph of groups} $\mathbf{Y}$ consists of  a graph $Y$, a group $Y_v$ for each $v\in V(Y)$, and a group $Y_e$ for each $e=\{v,w\}\in E(Y)$, together with  monomorphisms  $\varphi\colon Y_e \to Y_{i}$ $i=v,w$. 

Given a graph of groups $\mathbf{Y}$, one of the classic theorems of Bass-Serre theory provides the existence of a group $G=\pi_1(\mathbf{Y})$, called the \emph{fundamental group of the graph of groups} $\mathbf{Y}$ and the tree $T$(a graph with no cycles), called the \emph{Bass-Serre tree} of $\mathbf{Y}$, such that $G$ acts on $T$ without inversions, and the induced graph of groups is isomorphic to $\mathbf{Y}$. The identification $G=\pi_1(\mathbf{Y})$ is called a splitting of $G$.
\begin{definition}
Let $Y$ be a graph of groups with fundamental group $G$. The splitting  $G=\pi_1(Y)$ is \emph{acylindrical} if there is an integer  $k$ such that, for every path  $\gamma$ of length  $k$ in the Bass-Serre tree $T$ of  $Y$,  the stabilizer of $\gamma$ is finite. 
\end{definition}
Recall a \emph{geodesic line} of a simplicial tree $T$, is a simplicial embedding of $\mathbb{R}$ in $T$, where $\mathbb{R}$ has as vertex set $\dbZ$ and an edge joining any two consecutive integers.

\begin{theorem}\cite[Theorem 6.3]{MR4472883}\label{building:model:apply:haefliger}
Let $Y$ be a graph of groups with finitely generated fundamental group  $G$   and Bass-Serre tree  $T$. Consider the collection  $\mathcal{A}$ of all the geodesics of   $T$  that admit a  co-compact action of an infinite  virtually  cyclic subgroup of $G$. Then the space $\widetilde{T}$ given by the following  homotopy $G$-push-out

\begin{equation*}
\xymatrix{\displaystyle\bigsqcup_{\gamma\in \mathcal{A}}\gamma\ar[d] \ar[r]  & T\ar[d]\\
\displaystyle\bigsqcup_{\gamma\in \mathcal{A}} \{*_{\gamma}\}  \ar[r] & \widetilde{T} 
}
\end{equation*}
is a model $\widetilde{T}$ for $E_{\iso_{G}(\widetilde{T})}G$ where  $\iso_{G}(\widetilde{T})$ is the family generated by the isotropy groups of $\widetilde{T}$, i.e. by coning-off on $T$ the geodesics in $\mathcal{A}$  we obtain a model for  $E_{\iso_{G}(\widetilde{T})}G$. Moreover,  if  the splitting  $G=\pi_1(Y)$ is acylindrical, then the family  $\iso_{G}(\widetilde{T})$ contains the family  $\calF_n$ of $G$ for all $n\ge 0$.
\end{theorem}

The following theorem is mild generalization of \cite[Proposition 7.4]{MR4472883}. We include a proof for the sake of completeness.

\begin{theorem}\label{bass:serre:gd}
Let $Y$ be a graph of groups with finitely generated fundamental group $G$ and Bass-Serre tree $T$. Suppose that the splitting of  $G$ is acylindrical. Then for all $k\geq 1$ we have 

\[ \max \{ \gd_{\calF_k\cap G_v}(G_v), \gd_{\calF_k\cap G_e}(G_e)| v\in V(Y), e\in E(Y)\} \leq\gd_{\calF_k}(G)\]
and 
\[\gd_{\calF_k}(G)\leq \max \{ 2, \gd_{\calF_k\cap G_v}(G_v),  \gd_{\calF_k\cap G_e}(G_e)+1| v\in V(Y), e\in E(Y)\}\]
\end{theorem}
\begin{proof}
 For each $s\in V(Y)\cup E(Y) $ we have that $G_s$ is a subgroup of $G$, then the first inequality follows. 
 Now we prove the second inequality.  The splitting of $G$ is acylindrical, then we can use \cref{building:model:apply:haefliger} to obtain a 2-dimensional space $\widetilde{T}$ that is obtained from $T$ coning-off some geodesics of $T$, see \cref{tree}
 \begin{figure}[h]
    \centering
 \includegraphics[width=6.3in]{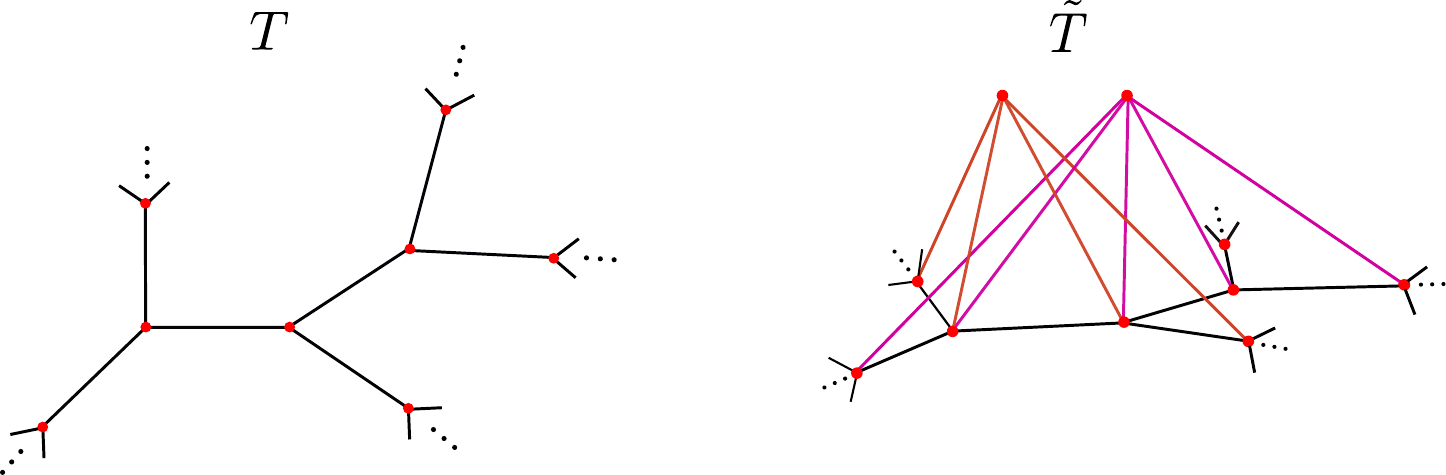} 
    \caption{\small Promoting $T$ to $\Tilde{T}$.}
    \label{tree}
\end{figure}
 , the space $\widetilde{T}$ is a model for $E_{\iso_{G}(\widetilde{T})}G$ and $\calF_k\subseteq \iso_{G}(\widetilde{T})$.  By \cref{prop:haefliger} we have \[\gd_{\calF_k}(G)\leq \max\{\gd_{\calF_k\cap G_\sigma}(G_{\sigma})+\dim(\sigma)| \ 
\sigma \text{ is a cell of } \ \widetilde{T} \}.\]

Let $\sigma$ be  a cell of $\widetilde{T}$, we compute $\gd_{\calF_k\cap G_\sigma}(G_{\sigma})+\dim(\sigma)$.
\begin{itemize}
    \item If $\sigma$ is 0-cell we have two cases $\sigma\in T$ or $\sigma\in \widetilde{T}-T$, in the first case we have $G_\sigma=G_v$ for some $v\in V(Y)$, in the other case we have $G_\sigma$ is virtually cyclic, then $\gd_{\calF_k\cap G_\sigma}(G_{\sigma})+\dim(\sigma)= \gd_{\calF_k\cap G_v}(G_v)$ or $0$.
    \item If $\sigma$ is 1-cell we have two cases $\sigma\in T$ or $\sigma$ has a vertex in $\widetilde{T}-T$, in the first case we have $G_\sigma=G_e$ for some $e\in E(Y)$, in the other case we have $G_\sigma$ is virtually cyclic, then $\gd_{\calF_k\cap G_\sigma}(G_{\sigma})+\dim(\sigma)= \gd_{\calF_k\cap G_e}(G_e)+1$ or $1$.
    \item If $\sigma$ is 2-cell, then $\sigma$ has a vertex in $\widetilde{T}-T$, then  $G_\sigma$ is virtually cyclic, it follows that $\gd_{\calF_k\cap G_\sigma}(G_{\sigma})+\dim(\sigma)= 2$.
\end{itemize}
 The inequality follows.
\end{proof}

\begin{proposition}
    Let $Y$ be a finite graph of groups such that for each  $v\in V(Y)$ the group $G_v$ is infinite finitely generated virtually abelian, with $\rank(G_e)<\rank(G_v)$. Suppose that the splitting of  $G=\pi_1(Y)$ is acylindrical. Let  $m=\max\{\rank(G_v)| v\in V(Y) \}$. Then for $1\leq k<m$ we have $\gd_{\calF_k}(G)=m+k.$ 
\end{proposition}
\begin{proof}
First, we prove that $\gd_{\calF_k}(G)\geq m+k$. The splitting of $G$ is acylindrical, then by \cref{bass:serre:gd} we have 
\[
\begin{split}
        \gd_{\calF_k}(G) & \geq \max \{ \gd_{\calF_k\cap G_v}(G_v), \gd_{\calF_k\cap G_e}(G_e)| v\in V(Y), e\in E(Y)\}\\
        & \geq \max \{ \rank(G_v)+k, \rank(G_e)+k| v\in V(Y), e\in E(Y)\}, \text{ from \cref{subgroup:Z:lower:bound:1}}\\
        &= \max \{ \rank(G_v)+k| v\in V(Y)\}, \text{$\rank(G_e)\leq \rank(G_v)$}\\
        &=m+k.
    \end{split}
\]
Also by \cref{bass:serre:gd} we have 
\[
\begin{split}
\gd_{\calF_k}(G)&\leq \max \{ 2, \gd_{\calF_k\cap G_v}(G_v),  \gd_{\calF_k\cap G_e}(G_e)+1| v\in V(Y), e\in E(Y)\}\\
&= \max \{ 2, \rank(G_v)+k, \rank(G_e)+k+1 | v\in V(Y), e\in E(Y)\}, \text{ from \cref{Virtually:dimension:Z}}\\
&= \max \{ \rank(G_v)+k | v\in V(Y)\}, \text{$\rank(G_e)<\rank(G_v)$ and $k\geq 1$}\\
&=m+k.
\end{split}
\]\end{proof}

\begin{corollary}
 Let $Y$ be a finite graph of groups such that for each $v\in V(Y)$ the group $G_v$ is infinite finitely generated virtually abelian, and for each $e\in E(Y)$ the group $G_e$ is a finite group.  Let  $m=\max\{rank(G_v)| v\in V(Y) \}$. Then for $1\leq k<m$ we have $\gd_{\calF_k}(G)=m+k.$ 
\end{corollary}
\bibliographystyle{alpha} 
\bibliography{mybib}
\end{document}

%% file: Preamble.tex
\usepackage{graphicx}
\usepackage{pinlabel} 
\usepackage{amsmath}
\usepackage{amsfonts}
\usepackage{amssymb}
\usepackage{mathtools}
\usepackage[all,cmtip]{xy}
\usepackage{amsthm}
\usepackage{tikz-cd}
\usepackage{comment}
\usepackage{enumerate}
\usepackage{url}
\usepackage{epsfig}
\usepackage[utf8]{inputenc}
\usepackage[capitalise]{cleveref}
\usepackage{geometry}
\makeatletter
\providecommand\@dotsep{5}
\def\listtodoname{List of Todos}
\def\listoftodos{\@starttoc{tdo}\listtodoname}
\makeatother

\newtheorem{theorem}{Theorem}[section]
\newtheorem{proposition}[theorem]{Proposition}
\newtheorem{corollary}[theorem]{Corollary}
\newtheorem{lemma}[theorem]{Lemma}

\newcommand{\mycomment}[1]{}

  \theoremstyle{definition}
\newtheorem{definition}[theorem]{Definition}

\newtheorem{remark}[theorem]{Remark}
\newtheorem{question}{Question}

\newcommand{\dbZ}{\mathbb{Z}}
\newcommand{\undZ}{\underline{\mathbb{Z}}}
\newcommand{\sub}{SUB}
\newcommand{\Z}{\mathbb{Z}}
\newcommand{\N}{\mathbb{N}}

\newcommand{\I}{I}

\newcommand{\calF}{{\mathcal F}}

\newcommand{\calG}{{\mathcal G}}
\newcommand{\calH}{{\mathcal H}}

\newcommand{\nbeq}{\begin{equation}}
\newcommand{\neeq}{\end{equation}}
\newcommand{\beq}{\begin{equation*}}
\newcommand{\eeq}{\end{equation*}}

\DeclareMathOperator{\cd}{cd}

\DeclareMathOperator{\vcd}{vcd}

\DeclareMathOperator{\rank}{rank}
\DeclareMathOperator{\gd}{gd}

\DeclareMathOperator{\Out}{Out}
\DeclareMathOperator{\Inn}{Inn}
\DeclareMathOperator{\Aut}{Aut}

\DeclareMathOperator{\cat}{CAT}

\DeclareMathOperator{\iso}{Iso}

\DeclareMathOperator{\Ima}{Im}
\DeclareMathOperator{\mor}{mor}

